\documentclass[12pt]{amsart}
\usepackage{latexsym,amssymb,amsmath}

\newcommand{\R}{{\mathbb R}}
\newcommand{\T}{{\mathcal T}}
\newcommand{\Z}{{\mathbb Z}}
\newcommand{\N}{{\mathbb N}}
\newcommand{\Q}{{\mathbb Q}}
\newcommand{\C}{{\mathbb C}}

\newcommand{\Om}{{\Omega}}
\newcommand{\La}{{\Lambda}}
\newcommand{\ti}{{\mathcal {T}}}

\newtheorem{theorem}{Theorem}[section]
\newtheorem{lemma}[theorem]{Lemma}
\newtheorem{proposition}[theorem]{Proposition}

\begin{document}

\title{On Fuglede's Conjecture for three intervals}

\author{Debashish Bose}
\address{Debashish Bose: Indian Institute of Technology Kanpur, India}
\email{debashishb@wientech.com}

\author{C.P. Anil Kumar}
\address{C.P. Anil Kumar: Infosys, Bangalore, India}
\email{anilkumar$\_$p@infosys.com}

\author{R.  Krishnan}
\address{R. Krishnan: Institute of Mathematical Sciences, Chennai, India}
\email{rkrishnan@imsc.res.in}

\author{Shobha Madan}
\address{Shobha Madan: Indian Institute of Technology Kanpur, India}
\email{madan@iitk.ac.in}
\maketitle

\begin{abstract}
In this paper, we first prove the {\it Tiling implies Spectral} part
of Fuglede's cojecture  for the three interval case. Then we prove
the converse {\it Spectral implies Tiling} in the case of three
equal intervals and also in the case where the intervals have
lengths $1/2,\, 1/4,\, 1/4$. Next, we consider  a set $\Omega
\subset \R$, which is a union of $n$ intervals. If $\Om$  is a
spectral set, we prove a structure theorem for the spectrum,
provided the spectrum is contained in some lattice. The method of
this proof has some implications on the {\it Spectral implies
Tiling} part of Fuglede's conjecture for three intervals. In the
final step of the proof, we need a symbolic computation using {\it
Mathematica}. Finally with one additional assumption we can conclude
that the {\it Spectral implies Tiling} holds in this case.

\end{abstract}

\section{\bf{Introduction}}

Let $\Omega$ be a Lebesgue measurable subset of $\R$ with finite
positive measure. For $\lambda \in \R$, let $$e_\lambda(x) =
\frac{1}{|\Om|^{1/2}}e^{2\pi i\lambda x}\chi_\Om(x),\,\,\ x\in \R.$$
\medskip

$\Omega$ is said to be a {\it{spectral set}} if there exists a
subset $\Lambda \subset \R$, such that the set $E_\Lambda =
\{e_\lambda : \lambda \in \Lambda \}$ is an orthonormal basis for
the Hilbert space $ L^2(\Omega)$, and then the pair $(\Omega ,
\Lambda)$ is called a {\it spectral pair.}
\medskip

We say that $\Omega $ as above {\it{tiles}} $\R$ by translations if
there exists a subset $\T \subset \R$, such that the set $\{\Omega +
t : t \in \T\}$, consisting of the translates of $\Omega$ by $\T$,
forms a partition a.e. of $\R$. The pair $(\Omega, \T)$ is called a
{\it tiling pair}. These definitions clearly extend to $\R^d,\,
d>1$.
\medskip

The study of the relationship between tiling and spectral properties
of measurable sets started with a conjecture proposed by Bent
Fuglede in 1974.
\medskip

{\bf {Fuglede's Conjecture}}. Let $\Omega$ be a measurable set in
$\R^d$ with finite positive measure. Then $\Omega$ is spectral if
and only if $\Omega$ tiles $\R^d$ by translations.
\medskip

Fuglede proved this conjecture in $\R^d$ under the additional
assumption that the spectrum, or the tiling set, is a
$d$-dimensional lattice [F]. In recent years there has been a lot of
activity on this problem. It is now known that, in this generality,
the conjecture is false in both directions if the dimension $d\geq
3$ ([T], [M], [KM2]) and ([KM1], [FR], [FMM]).  However, the
conjecture is still open in all dimensions $d \geq 3 $ under the
additional hypothesis that the set $\Omega$ is a convex set. For
convex sets, the conjecture is trivial for $d = 1$, and for $d =2$,
it was proved in [IKT1], [IKT2] and [K]. In dimension $1$, the
problem has been shown to be related to some number-theoretic
questions, which are of independent interest and many partial
results supporting the conjecture are known (see, for example [LW2],
[L2], [K2], [PW]).
\medskip

In this paper, we restrict ourselves to one dimension and to the
case when the set $\Omega$ is a union of three intervals. This work
was inspired by the paper [L1] by I. Laba, where Fuglede's
conjecture is proved for the case that $\Omega$ is a union of two
intervals. We state Laba's Theorem here in order to put the main
result of this paper in perspective.
\medskip

{\bf Theorem [L1].} Let $\Om = [0,\,r) \cup[a,\, a+1-r)$, with $0 <r
\leq 1/2, \,\, a \geq r$. Then the following are equivalent;
\begin{enumerate}
    \item $\Om$ is spectral.
    \item Either (i) $ a-r \in \Z$, or (ii) $r = 1/2, \, a = n/2$ for some $n \in \Z$.
    \item $\Om$ tiles $\R$.
\end{enumerate}
Further, when $\Om$ is spectral, then $\La = \Z$ if 2(i) holds, and
$\La = 2\Z \cup (2\Z + p/n)$ for some odd integer $p$, if 2(ii)
holds.
\medskip

In section 2, we prove that if $\Om$ is a union of three intervals,
then ``Tiling implies Spectral''. This proof, though somewhat long,
uses elementary arguments and some known results. We observe that
the occurrence of certain patterns in the tiling imposes
restrictions on the lengths of the intervals. This allows us to
identify the different cases to be considered, and we prove that
$\Om$ is spectral in each case.
\medskip

In section 3, we consider two particular cases where $\Om = A \cup
B\cup C$ with $A,\,B,\,C$ intervals, and either $|A| = |B| = |C| =
1/3$, or $|A| = 1/2, \,\, |B| = |C| = 1/4$, and prove "Spectral
implies Tiling" in these cases. Here, orthogonality conditions
impose restrictions on the end-points of the intervals, and we use a
powerful theorem on tiling  of integers due to Newman [N] to
conclude that $\Om$ tiles $\R$. (Newman uses the word tesselation
for tiling of integers in his paper).
\medskip

In Section 4 we briefly digress to the case of $n$-intervals, and
obtain information about the spectrum for such spectral sets. We
assume that $\Omega =\cup_{j=1}^n [a_j,a_j +r_j)$, with
$\sum_{j=1}^n r_j =1$ and $a_j + r_j < a_{j+1}$, is spectral with a
spectrum $\La$. Note that if $\La$ is a spectrum for $\Om $, then
any translate of $\La$  is again a spectrum for $\Om$. In this paper
we always assume that
$$0 \in \La \subset \La -\La. $$
Further, by the orthogonality of the set $E_\La$, we have $$0 \in
\La \subset \La -\La  \subset \Z_\Om,$$ where $\Z_\Om$ stands for
the zero set of of the Fourier transform of the indicator function
$\chi_\Om$, along with the point $0$, i.e.,
$$\Z_\Om = \{ \xi \in \R : \widehat{\chi_\Om}(\xi) = 0 \} \cup \{0\}.$$
In our investigation, the geometry of the zero set $\Z_\Om$ will
play an important role, as also a deep theorem due to Landau [L]
regarding the density of sets of interpolation and sets of sampling.
If $(\Om, \La) $ is a spectral pair, then Landau's theorem applies
and says that the asymptotic density $\rho(\La)$ of $\La$ equals $
1/|\Om|$, where the asymptotic density is given by $\rho(\La):=
\lim_{r\rightarrow\infty} \frac{card(\La\cap[-r,r])}{2r}.$
\medskip

In the literature, it is generally assumed that $\La$ is contained
in some lattice $\mathcal L$. Since $\La$ has positive asymptotic
density, by Szemer$\grave{e}$di's theorem [S], $\La$ will contain
arbitrarily long arithmetic progressions (APs). In Theorem 4.3, we
prove that if $\Om$ is a union of $n$ intervals and $\La$ contains
an AP of length $2n$, then $\La$ must contain the complete AP. As a
consequence, we show that $\Om \,\,d-$tiles $\R$, where $d$ is the
common difference of the AP. So we need to search for APs of length
$2n$ in $\La$. In Lemma 4.4, we show that even if $\La -\La$ is just
$\delta$-separated (instead of being contained in a lattice), then
$\La$ will contain APs of arbitrary length.
\medskip

We return to three intervals in Section 5, with the assumption that
$\La - \La$ is $\delta$-separated. It then turns out that either (i)
$\Om/\Z \approx [0,1]$, or (ii)  $\Om/{2\Z} \approx [0,1/2] \cup
[n/2, (n+1)/2]$, for some $n$, or  (iii) the case is that of equal
intervals (not necessarily $3$ equal intervals!). In the first two
cases ``Spectral implies Tiling''  follows from [F] and [L1],
respectively. The third case is rather complex, and assuming that
$\La \subset \mathcal L$, we are lead to questions about vanishing
sums of roots of unity. Then we use symbolic computation using
Mathematica in an attempt to resolve Fuglede's conjecture for three
intervals. With one additional assumption on the spectrum, we show
``Spectral implies Tiling''. The analysis for this computation is
given in Section 6.

\section{\bf{Tiling implies Spectral}}

Let $A,\,B,\,C$ be three disjoint intervals in $\R$. In this section
we prove the following theorem:

\begin{theorem}
Let $\Om = A \cup B \cup C$, $|A|+|B|+|C|=1$.
If $\Om$ tiles $\R$ by translations, then $\Om$ is a spectral set.
\end{theorem}

For the proof, we adopt the following notation for
convenience:\\Whenever we need to keep track of the intervals $A,\,
B,\, C$ as being part of a translate of $\Om$ by $t$, we will write
$\Om_t = \Om + t = A_t \cup B_t \cup C_t $.
\medskip

We begin with a simple lemma.

\begin{lemma}
Suppose $\Om$ tiles $\R$. Suppose that in some  tiling
by $\Om$, $AA$ occurs, then either $|A|=\frac{1}{2}$ or $|A| = |B| = |C| = \frac{1}{3}$.
\end{lemma}

\begin{proof}
Suppose not all three intervals are of equal length, and there
exists a set $\mathcal T$ such that $(\Om, \mathcal T)$ is a tiling
pair. If for some $s,t \in \ti$, $A_sA_t$,  occurs, then $|A|\geq
|B|,|C|$.

Suppose $BB$ also occurs, then $|B| \geq |A|$, so $|A|=|B|< 1/2$,
since $|C| > 0$. Now $CC$ cannot also occur (for then
$|A|=|B|=|C|$). Hence there is a gap between $C_s$ and $C_t$ which
equals $ 0< |A| - |C| < |A|,\, |B|$, so it cannot be filled at all.

Hence if $AA$ occurs, neither $BB$ nor $CC$ can also occur in that
tiling. Now the gap between  $B_s$ and $B_t$ is of length $|A| - |B|
< |A|$, and no $B$ can lie in this gap, so it can be filled only by
a $C$. It follows that  $|A|-|B|=|C|$  and so $|A|=\frac{1}{2}$.
\end{proof}

\medskip

This leads us to consider the following cases:
\begin{description}
 \item[Case 1] $|A|,|B|,|C|\neq \frac{1}{2}$ and not all are equal:
        \begin{description}
                \item[(1a)] No two of $|A|,|B|,|C|$ are equal.
            \item[(1b)] $|A|=|B|\neq|C|.$
        \end{description}
\medskip
 \item[Case 2] $|A|= \frac{1}{2}$:
            \begin{description}
                \item[(2a)]$|B|\neq|C|$.
                \item[(2b)]$|B|=|C|=\frac{1}{4}$.
            \end{description}
\medskip
 \item[Case 3] $|A|=|B|=|C|=\frac{1}{3}$.
\end{description}
\bigskip
{\bf Proof of Theorem:}
We will prove that Tiling implies Spectral, in each of the above  cases.
\medskip

{\bf Case(1a)} : $|A|,|B|,|C|$ are all distinct and none equals  $1/2$.

We claim that in this case the tiling pattern
has to be of the form
 $$---ABC|ABC|ABC---$$
 or
 $$---ACB|ACB|ACB---$$
Then $\Omega$ tiles $\R$ by $\Z$, hence is spectral \cite{F}. Since
we already know that $XX$ cannot appear, our claim will be proved if
we show that no three consecutive intervals appear as $XYX$.

Suppose $A_tBA_s$ occurs. No $C$ can appear between $C_t$ and $C_s$,
and the gap between the  $C_t$ and $C_s$ is of length
$|A|+|B|-|C|<|A|+|B|$ and so cannot be filled by $AB$, (nor by $AA$,
nor by $BB$), hence equals either $|A|$ or $|B|$. But then, either
$|B|=|C|$ or $|A|=|C|$, a contradiction.

\medskip

 {\bf Case(1b)}: $|A|=|B|\neq|C|,\,\,|C|\neq 1/2$.

 We already know that none of $AA,BB,CC$ can occur in
 the tiling.
 We show below that $ABA$ and $BAB$ cannot occur:
 Suppose $A_sBA_t$ occurs,  then $C_s$ and $C_t$ are consecutive $C$'s
 and the gap between them
 has length $|A|+ |B|-|C|=2|A|-|C| = 4|A| -1$.
 This gap has to be filled by $A$'s and $B$'s.
 Let $4|A|-1= m|A|$. Clearly $m < 4$, and we check easily that this
 leads to a contradiction ($m=0 \Rightarrow |C|= 1/2$;
 $m= 1 \Rightarrow |A|=|B|=|C|= 1/3$, $m=2 \Rightarrow |A|=|B|= 1/2$
and $m=3  \Rightarrow  |A|=1$). Hence $A_sBA_t$ cannot occur.
\medskip

This means that in any tiling, the gap between two consecutive $C$'s
is filled by at most two of $A$ and $B$. In other words, $C$ is
translated by at most 1; hence also $A$ and $B$. But if none of the
translates is greater than 1, then $CAC,\, CBC,\, BCB, \,$ and $ACA$
are excluded (for an $A_sCA_t$ occurs iff $B_sCB_t$ occurs, and also
$A_sCA_t$ occurs iff either $C_sAC_t$ occurs or $C_sBC_t$ occurs;
but an occurrence of $A_sCA_t$ would imply that $B$ is translated by
$|B|+2|A|+|C| > 1$, which is not possible). Hence, the tiling
pattern must be either
$$---ABCABCABC---$$
or
$$---ACBACBACB---$$
i.e. $\Omega$ tiles $\R$ by $\Z$, and hence is spectral by [F].

\medskip

{\bf Case (2a)} : $|A|= 1/2, \,\,|B|\neq|C|$.

By Lemma 2.2, $BB$ and $CC$ cannot occur in the tiling. Next we show
that none of $BAB,\,ABA,\,CAC,\,ACA$ can occur.

First note that $BAB$ occurs iff $ABA$ occurs: for, if $BAB$ occurs
as, say $B_sA_rB_t$, where $s, r, t \in \T$, then clearly $t-s = |A|
+|B|$. Consider the corresponding tile translates,
$$\Om_s = A_s \cup B_s \cup C_s, \,{\mbox{and }} \, \Om_t = A_t \cup
B_t \cup C_t .$$
Then the gap between $A_t$ and $A_s$ is of length $|B|$. This gap
cannot be filled by $A$, since $|A| > |B|$, nor can it be filled by
$C$, since $|C| \neq |B|$, and two consecutive $C$'s are ruled out
anyway. So this gap can be filled up by a single $B$, resulting in
$A_sBA_t$, an $ABA$ pattern. The proof of the converse is similar.

Now, suppose $A_sBA_t$ occurs. The gap between $C_s$ and $C_t$ is of
length $|A|+|B|-|C|<|A|+|B|$. This gap has to be filled by a single
$A$ or a single $B$; and this would imply $|B|=|C|$ or $|A| = |C| =
1/2$. Similarly  $CAC$ and $ACA$ cannot occur.
\medskip

Next, suppose that there is a string $B_sAA...AAB_t$ in the tiling
with $n$ consecutive $A$'s between $B_s$ and $B_t$. The gap between
$C_s$ and $C_t$ is then of length $n/2+|B|-|C|$ and has to be filled
by $A$'s and $B$'s. But $B$'s can occur just after the  $C_s$ and
just before the
 $C_t$, otherwise $ABA$ will occur. If no $B$ occurs, then $|B|=|C|$ and if
a single B occurs then $|C|=0$ or $1/2$. Hence there
exists a string
$$---C_sBAA---ABC_t--$$
with $m$ consecutive $A$'s with a $B$ on either side, lying between
$C_s$ and $C_t$. Then $\frac{n}{2}+|B|-|C|=2|B|+\frac{m}{2}$, for
some integer $m$, which means $m=n-1$. We can repeat the argument
$n$ times to to get that $BAB$ must occur somewhere, which is not
possible. Similarly $CAA---AC$ does not occur.

Thus the only possibilities are tiling patterns of the form :
\begin{enumerate}
\item $ ---AA---A\left|BC\left|BC---\right|BC\right|AA---ABC---$
\item $ ---AA---A\left|BC\left|BC---\right|BC\right|BA---ACB---$
\end{enumerate}
and $(1\acute)$ ,$(2\acute)$ with $B$ and $C$ interchanged. We show
that $(2)$ and $(2\acute) $ cannot happen. Suppose $(2)$ occurs;
consider the part consisting of a sequence of $n \,\,BS$'s followed
by a $B$ between $A_s$ and $A_t$, i.e. $A_sBC---BCBA_t$ and the
$C_s,\,\,C_t$ corresponding to these $A$'s. The gap is of length
$\frac{n}{2}+\frac{1}{2}+|B|-|C|$ and has to be filled by $A$'s and
$B$'s. But this leads to a contradiction as shown above. So finally
the tiling pattern has to be of the form $(1)$ or $(1\acute )$.

We write $\Om =[0,1/2)\cup[b,b+r)\cup[c,c+1/2-r).$ It is easy to see
that the above tiling pattern (namely $(1)$ or $(1)'$) for $\Om$,
say by a tiling set $\ti$ implies that both $\Om
_1=[0,1/2)\cup[b,b+1/2)$ and $\Om _2=[0,1/2)\cup[c-r,c-r+1/2)$ tile
$\R$ by the same tiling set $\ti$. (Alternatively we may have to
work with the sets $\Om _1'=[0,1/2)\cup[b+r -1/2,b+r), \,\, \Om
_2'=[0,1/2)\cup[c,c+1/2)$). Using the result for two intervals [L1],
this implies that $b=\frac{n}{2},\,c-r=\frac{k}{2}$ with $n,k \in
\Z$ . Then $c-r-b = \frac{k-n}{2}$. Hence $\Om_1 = [0,1/2)\cup[n/2,
n/2+1/2)$. Clearly $n$ is a period for the tiling set $\mathcal T$,
but $n/2$ is not. So if $k_0$ is the period of the tiling then
$k_0|n$, but $k_0 \not| \frac{n}{2}$. Hence if $j $ is the largest
integer such that $2^j|n$, the same is true for $k_0$. Therefore,
$\frac{n}{k_0}\in 2\Z+1$. Also $k_0|k$, and  $\frac{k}{k_0} \in
2\Z+1$ and so $k_0 | l$, where $l = c-r-b = \frac{k-n}{2}$, and $l
\in \Z$.

We show finally that the set $\Lambda=2\Z \cup (2\Z+\frac{1}{k_0})$
is a spectrum for $\Om$ (we have taken $p=\frac{n}{k_0}$ in Laba's
theorem as stated in the Introduction). To check orthogonality, note
that if $\lambda \in \La - \La, \,\, \lambda \neq 0$, then

\begin{eqnarray*}
\int_\Om e^{2\pi i\lambda\xi}d\xi & = & \int_0^{1/2}+\int_b^{b+r} +
\int_c^{c+1/2-r} e^{2\pi i\lambda\xi}d\xi\\
 & = & \int_0^{1/2} e^{2\pi i\lambda\xi}d\xi+\int_b^{b+r} e^{2\pi
 i\lambda\xi}d\xi +
\int_{b+r+l}^{b+1/2+l} e^{2\pi i\lambda\xi}d\xi \\
 & = & \int_0^{1/2} e^{2\pi i\lambda\xi}d\xi + \int_b^{b+r}
 e^{2\pi i\lambda\xi}d\xi
+ e^{2\pi i\lambda l} \int_{b+r}^{b+1/2}
 e^{2\pi i\lambda\xi}d\xi\\
 & = & \int_{\Omega_1} e^{2\pi i \lambda \xi} d\xi  = 0
\end{eqnarray*}
since $k_0|l$ and since $\La$ is   a  spectrum  for $ {\Om_1}$. As
in [L1], it is easy to see that $\La$ is complete. Alternatively,
this follows from [P1] and [LW].
\medskip

{\bf Case 2b} is a special case of 4 equal intervals, and {\bf Case
3} is 3 equal intervals. In each of these cases tiling implies
spectral follows by more general results [L2]. However, these cases
can be handled by using a theorem due to Newman \cite{N}, and we do
this in the next section.
\bigskip

\section{\bf{The Equal Interval Cases}}

In this section, we will use the fact that spectral and tiling
properties of sets are invariant under translations and dilations.
For convenience, we scale the set $\Om$ suitably and prove both
implications of Fuglede's conjecture for the two sets  $\Om_3 = [0,
1] \cup [a, a+ 1] \cup [b, b+ 1]$ and $\Om_4 = [0,2] \cup [a, a+1]
\cup [b, b+1]$.

An essential ingredient of our proofs will be the following theorem
on tiling of integers.
\begin{theorem}\emph{(Newman)}\label{N}
Let $A = \{a_1,a_2,...,a_k\}$ be distinct integers with
$k=p^{\alpha}$, $p$ a prime and $\alpha$ a positive integer. For
each pair $a_i,a_j$, $i\neq j$, let $e_{ij}$ denote the largest
integer so that $p^{e_{ij}}|(a_i-a_j)$, and let $S = \{e_{i,j}: i,j
= 1,2,...,k;\, i \neq j\}$. Then the set $A$ tiles the set of
integers $\Z$ iff the set $S$ has at most $\alpha$ distinct
elements.
\end{theorem}
\medskip

\subsection{ Tiling implies Spectral.} Without loss of generality
we may assume that $0 \in \mathcal T$, where  $\mathcal T$ is the
tiling set. Then in the cases under consideration, the end-points of
the intervals will be integers and we are in the context of tiling
of $\Z$.
\medskip

{\bf Case 3:} We apply Newman's Theorem to the set $A = \{
0,\,a,\,b\}$,  $p=3,\,\alpha =1$, hence $S$ must be a singleton. If
$a = 3^j n$ and $b = 3^k m$ with $n,\,m$ not divisible by $3$, we
see that $ \{j,\,k\} \subset S$. It follows that $j = k$ and  $m-n$
is not divisible by $3$. We may write $a = 3^j(3r+1)$ and $b=
3^j(3s+2); \,\, r, s \in \Z$. But then, we can check easily that the
set $\La = \Z \cup (\Z + \frac{1}{3^{j+1}}) \cup (\Z +
\frac{2}{3^{j+1}})$ is a spectrum for $\Om_3$.

To check completeness, suppose that $f \in L^2(\Om)$ satisfies
$\left\langle f, e^{2\pi i \lambda .}\right\rangle = 0, \,\,\,
\forall \lambda \in \La$. With $\omega = e^{2\pi i/3}$, define
$1-$periodic functions by $ f_1^\#, \, f_2^\#, \, f_3^\#$ by

\begin{equation}
\left(\begin{array}{c}
f_1^\#(x)\\f_2^\#(x)\\f_3^\#(x)\\
\end{array} \right) =
\left(\begin{array}{ccccc}
1 & 1  & 1 \\
1 & \omega & \omega^2\\
1 &\omega^2 & \omega\\
\end{array}\right)
\left(\begin{array}{c}
f(x)\\f(x+ a)\\f(x+b)\\
\end{array}\right)
\end{equation}
where $x \in [0, 1]$.

Then check that $\left\langle f, e^{2\pi i \lambda .}\right\rangle =
0, \,\,\forall \lambda \in \La \,\, \Longleftrightarrow
\,\,\left\langle f_j^\#, e^{2\pi i k .}\right\rangle = 0,
\,\,\forall k \in \Z, \,\,\, j=1,2,3 \,\,\Longleftrightarrow \,\,
f_j^\# = 0 ,\,\, j=1,2,3 \,\, \Longleftrightarrow \,\,f = 0$.
\medskip

{\bf Case 2b:} This time we use Newman's theorem for the set $A =
\{0,1, a, b\}$, $p = 2,\, \alpha =2$, hence $S$ can have at most $2$
elements. Let $a = 2^j(2n+1), \, b = 2^k(2m+1)$.
Then the following cases arise:\\
1. If $j\neq 0,\, k\neq 0$, $\{0,j,k\} \subset S$. Hence $j=k$,
but then $a-b= 2^j.2(n-m)$, so $j+1 \in S$, and the set $S$ still
has cardinality $\geq 3$, therefore $A$ cannot tile $\Z$.\\
2. Suppose that $j=0, \, k\neq 0$. Then $A = \{0,1, a = 2n +1,\,
 b = 2^k(2m+1)\}$. If $n = 2^{l-1}r $ with $l \geq 1 $ and  $r$
 an odd integer, then $S = \{0, k, l\}$. Since $S$ has at most
 $2$ elements, we get $k = l$ in this case. \\
3. Let $j,\,k=0$. We see then that $S$ has at least three distinct
$e_{i,j}$'s, so $A$ cannot tile $\Z$.

We conclude that if $\Om_4$ tiles, then $a = 2^l r +1,\,\, b= 2^l
s$, where $r,s \in 2\Z+1$ and then it is easy to see that the set
$\La = \Z \cup \Z + \frac{1}{2^{l+1}}$ is a spectrum. This completes
the proof of the ``Tiling implies Spectral'' part of Fuglede's
conjecture for $3$ intervals.

\medskip

\subsection {Spectral implies Tiling:}
We now proceed to prove the converse for the two cases $\Om_3$ and $\Om_4$.

In the first case, the proof is long, but we are able to construct
the spectrum as a subset of the zero set of $\widehat{\chi_\Om}$,
using only the fact $\rho(\La) = |\Om|$. In the process we get
information on the end-points of the three intervals, so that
Newman's theorem can be applied to conclude tiling.

{\bf Case 3:} Let $\Omega_3 = [0, 1]\cup [a,a+1]\cup[b,b+1]$.
Without loss of generality, let  $ 0 \in \Lambda \subset
\Lambda-\Lambda \subset \Z_\Omega$, then
$$2 \pi i \lambda  (\widehat{\chi_\Omega}(\lambda)) = e^{2 \pi i \lambda} -1 +
 e^{2 \pi i \lambda (a+1)}-e^{2 \pi i \lambda a} +
 e^{2 \pi i \lambda (b+1)} - e^{ 2 \pi i \lambda b}$$
$$ = ( e^{2 \pi i \lambda}-1)
(1+ e^{2 \pi i \lambda a} + e^{2 \pi i \lambda b})$$ Hence, if
$\lambda \in \Z_\Omega$, then either  $ \lambda \in \Z \mbox {  or }
1+ e^{2 \pi i \lambda a} + e^{2 \pi i \lambda b}= 0.$

Let $$\Z_\Omega^1=\{\lambda: e^{2 \pi i \lambda a}
=\omega,\, e^{2 \pi i \lambda b}=\omega^2\}, $$
$$\Z_\Omega^2=\{\lambda: e^{2 \pi i \lambda a} =
\omega^2,\, e^{2 \pi i \lambda b}=\omega\}. $$
\medskip
We collect some easy facts in the following lemma:\\
\medskip
{\bf Lemma.} Let $\lambda_1 \in \Z_\Omega^1,\,\, \lambda_2 \in
\Z_\Omega^2$, then the following hold:
\begin{enumerate}
\item $-\lambda_1,\,\, 2\lambda_1  \in \Z_\Om^2$ and  $-\lambda_2,\,\,
2\lambda_2  \in \Z_\Om^1$.
\item $\lambda_1 - \lambda_2 \in \Z_\Om^2$  and  $-\lambda_1 +
\lambda_2 \in \Z_\Om^1$.
\item Let $\alpha \in \Z_\Om^1 \cup \Z_\Om^2$, then $\lambda_1 +
3\alpha\Z \subset \Z_\Om^1,\,\,\, \lambda_2 + 3\alpha\Z \subset \Z_\Om^2$.
\item If $\alpha_\circ $ is the smallest positive real
number in $\Z_\Om^1 \cup \Z_\Om^2$, say $\alpha_\circ \in \Z_\Om^1$,
then  $\Z_\Om^1 = \alpha_\circ + 3\alpha_\circ \Z$  and  $\Z_\Om^2 =
2\alpha_\circ + 3\alpha_\circ \Z$.
\end{enumerate}
\begin{proof}
It is easy to verify (1), (2) and (3). We need to prove (4): First
note that since $\widehat{\chi_\Om}(0) = 3$, so  $\widehat{\chi_\Om}
> 0$ in a neighbourhood of $0$, and there exists a smallest positive
real number, say  $\alpha_\circ \in \Z_\Om^1 \cup \Z_\Om^2$. We
assume that $\alpha_\circ \in \Z_\Om^1 $. By (3), we only need to
prove that $\Z_\Om^1 \subset \alpha_\circ + 3\alpha_\circ\Z$.
Suppose not, then there exists $\beta \in \Z_\Om^1$,  such that
$\beta \notin \alpha_\circ + 3\alpha_\circ \Z$. Now by $(3),\,
\beta+3\alpha_\circ \in \Z_\Om^1$,  so we may assume that $\beta \in
(\alpha_\circ, 3\alpha_\circ]$. But $ 2\alpha_\circ, 3\alpha_\circ
\notin \Z_\Om^1$ and, hence the only possibility is $\beta \in
(\alpha_\circ, 2\alpha_\circ)\cup (2\alpha_\circ, 3\alpha_\circ)$.
Now in case $\beta \in (\alpha_\circ, 2\alpha_\circ)$, we use (1)
and (2) to get $2\alpha_\circ - \beta \in \Z_\Om^1$.  But $0 <
2\alpha_\circ - \beta < \alpha_\circ$, which contradicts the
minimality of $\alpha_\circ$. In the other case, that $\beta \in
(2\alpha_\circ, 3\alpha_\circ)$, by (1) and (3) we get
$3\alpha_\circ -\beta \in \Z_\Om^2$. But then $0 < 3\alpha_\circ
-\beta < \alpha_\circ$, again a contradiction. Hence $\Z_\Om^1 =
\alpha_\circ + 3\alpha_\circ \Z$. By similar arguments,  we get
$\Z_\Om^2 = 2\alpha_\circ + 3\alpha_\circ \Z.$
\end{proof}

%Let $\lambda_1 \in \Z_\Om^1 $ , then $\lambda_1 - 2\alpha_\circ \in \Z_\Om^2$
%since $2 \alpha_\circ \in \Z_\Om^2$
%$\Rightarrow \lambda_1 - \alpha_\circ \in 3\alpha_\circ \Z $

%As $ e^{2 \pi i a (\lambda_1 - \alpha_\circ)} e^{- 2 \pi i a \alpha_\circ}=\omega^2 $
%$\Rightarrow e^{2 \pi i a (\lambda_1 - \alpha_\circ) } = 1 = {e^{2 \pi i a \alpha_\circ}}^{3k}$
%$\Rightarrow \lambda_1 = \alpha_\circ + 3 \alpha_\circ k , k \in \Z $
Therefore,
$$ \Z_{\Omega} = \Z \cup \Z_\Om^1 \cup \Z_\Om^2
 = \Z \cup (\alpha_\circ + 3 \alpha_\circ \Z) \cup (2\alpha_\circ
+ 3 \alpha_\circ \Z).$$ Note that if $\lambda_1, \lambda_2 \in \Z$
then $e_{\lambda_1},\, e_{\lambda_2}$ are mutually orthogonal. But
$\La$ must have upper asymptotic density $3$, by Landau's density
theorem, hence $\rho(\La \setminus \Z) = 2$. In other words
$\rho(\La \cap (\Z_\Om^1 \cup \Z_\Om^2)) \geq 2$.

Our next step is to actually find all of $\La$ as a subset of $\Z_\Om$.

Suppose first $\La \cap \Z_\Om^1 \neq \emptyset$ and that
$\lambda_1$ is the smallest positive element of this set. Now if $
\lambda_1' \in \La \cap \Z_\Om^1$, then $\lambda_1 - \lambda_1 ' \in
\La - \La \subset \Z_\Om$, and by the definition of $\Z_\Om^1$ and
$\Z_\Om^2$, we see that
$$e^{2\pi i (\lambda_1 -\lambda_1')a} = 1 = e^{2\pi i (\lambda_1 -\lambda_1')b}$$
This means that $\lambda_1 -\lambda_1' \in \Z$. Therefore $\La \cap
\Z_\Om^1  \subset \lambda_1 + \Z$. A similar argument shows that
$\La \cap \Z_\Om^2 \subset \lambda_2 + \Z$, where $\lambda_2 \in
\La$. Therefore,
$$\rho(\La \cap \Z_\Om^1) \leq 1, \hspace{.5cm} \rho(\La \cap \Z_\Om^2) \leq 1$$
With this bound on the density, not only $\La \cap\Z_\Om^1 \neq
\emptyset$ and $\La \cap \Z_\Om^2 \neq \emptyset$, but each of these
sets must contribute a density $1$ to $\La$.  Therefore, from the
set $\lambda_1 + \Z$, two consecutive elements, say $\lambda_1 + n$
and $\lambda_1 + (n+1)$ must lie in $\La \cap \Z_\Om^1$, so their
difference, namely $1$, satisfies
$$e^{2\pi i a} = 1 = e^{2\pi i b}$$
and so $a,\,b \in \Z$, and
\begin{eqnarray*}
\lambda_1 + \Z & \subset & \Z_\Om^1  =  \alpha_\circ +3\alpha_\circ \Z \\
\lambda_2 +\Z & \subset & \Z_\Om^2  =  2\alpha_\circ +3\alpha_\circ \Z
\end{eqnarray*}
Finally we see that $\alpha_\circ \in \Q$, since $e^{2\pi i \alpha_\circ a} = \omega$, and $a \in \Z$. Putting $\alpha_\circ = p/q$ and $\lambda_1 = \alpha_\circ +3\alpha_\circ k_1, \,\, k_1 \in \Z$, we see easily that $\Z \subset \frac{p}{q}(3\Z)$, which means $p=1$ and $q = 3q_1$ for some $q_1 \in \Z$. But then
$$a = (3n+1)q_1, \hspace{.5cm} b = (3m +2)q_1. $$
We conclude that $\Om_3$ tiles, by applying Newman's Theorem to the
set $A = \{ 0, \, (3n+1)q,\, (3m +2)q \}, \,\, p=3, \, \alpha = 1.$
\medskip

 %{3\alpha_\circ(n-m),\,\,n,m \in \Z$ and also $e^{2\pi i(\lambda -\lambda ')a} =1$, which means that $\lambda -\lambda' \notin \Z_\Om^1 \cup \Z_\Om^2$. Hence $ \lambda -\lambda' \in 3\Z$. With a similar argument for $\La \cap \Z_\Om^2$, we conclude that
 %$\lambda'$, then $\lambda - \lambda' \in \La -\La \subset \Z_\Om$ can be orthogonal is if $\lambda -\l
 %ambda '\in 3\Z$. We also get that $a,b \in \Z/3$. Now since $\alpha_\circ a \in \Z +1/3$, we must have $\alpha_\circ \in \Q$, say $\alpha_\circ = %\frac{p}{q}$ with $(p,q) =1$ . We conclude that $$(\lambda_1+3\Z) \subseteq (\alpha_\circ + 3 \alpha_\circ \Z),\,\,(2\lambda_1+3\Z) \subseteq(2\alpha_\circ + 3 \alpha_\circ \Z),$$
%and infact,
%$$  \Lambda \subseteq 3\Z \cup (\lambda_1+3\Z) \cup (2\lambda_1 + 3\Z)\subseteq \Z_\Om $$

%Putting $\lambda_1 = \alpha_\circ + 3\alpha_\circ k_\circ$, for some $k_\circ \in \Z$, we see easily that $\frac{3p}{q}\Z + 3\Z \subseteq \frac{3p}{q}\Z$. But this is possible only if $p =1$. It follows that $ a= \frac{3n+1}{3}q,\,\, b=\frac{3m + 2}{3} q$. We end the proof by a simple application of Newman's Theorem to the three integers $0,(3n+1)q, (3m+2)q$, with $p=3,\, \alpha = 1$, to conclude that $\Om_3$ tiles $\R$.
%\end{proof}
%\end{document}}
{\bf Case 2b:} In the case $\Om_4$, we use a theorem of [JP], to
conclude that if $\Om_4$ is spectral, then the end-points of the
intervals lie on $\Z$. For $\Om_3$, we deduced this without invoking
this theorem.

%\begin{proposition}
%Let $\Omega = A\cup B \cup C ; \,\,|A|=\frac{1}{2},\,|B|=\frac{1}{4},\,|C|=\frac{1}{4}  $. Suppose $\Omega$ is a spectral set. Then $\Omega$ tiles $\R$ by translations.
%\end{proposition}
 The spectrum $\La$ will have density $= 4$, by Landau's Theorem.
 It follows from Corollary 2.4.1 of [JP] that $a,b \in \Z$, so that
 $\Om_4 = [0,2]\cup [M,M+1] \cup [N,N+1]$ with $M,\,N \in \Z$, and so
$$2 \pi i \lambda (\widehat{\chi_\Om}(\lambda)) = (e^{2\pi i\lambda}-1)(1+ e^{2\pi i\lambda} +  e^{2\pi i\lambda M} + e^{2\pi i\lambda N}).$$

Hence
$$\Z_\Om = \Z  \cup \Z_\Om^1 \cup \Z_\Om^2 \cup \Z_\Om^3,$$
where
$$\Z_\Om^1 = \{\lambda : 1+ e^{2 \pi i \lambda } = 0, \,e^{2\pi i \lambda N}+ e^{2 \pi i \lambda M}=0 \},$$
$$\Z_\Om^2 = \{\lambda : 1+ e^{2 \pi i \lambda N} = 0,\, e^{2\pi i \lambda M}+ e^{2 \pi i
\lambda} =0 \},$$
$$\Z_\Om^3 = \{\lambda : 1+ e^{2 \pi i \lambda M} =0,\, e^{2\pi i \lambda N}+ e^{2 \pi i \lambda }=0 \}.$$

We collect some facts, which are easy to verify:
\begin{enumerate}
\item If $\lambda \in \Z_\Om^i$, then $\lambda + \Z \subset \Z_\Om^i, \,\, i = 1,2,3.$ Further if $\lambda \in \La$, then $\lambda + \Z \subset \La$, i.e. $\La$ is 1-periodic.
\item If $\lambda \in \Z_\Om^i$, then $ - \lambda \subset \Z_\Om^i, \,\, i = 1,2,3.$
\item If $\Z_\Om^1 \neq \emptyset$, then $M-N$ is an odd integer, and $\Z_\Om^1 = \Z + 1/2$. Without loss of generality, we assume that $N = 2n+1$ and $M = 2m$. Thus, in this case $\Z_\Om^1 \subset \Z_\Om^2$.
\item If $\lambda, \, \lambda' \in  \Z_\Om^2$ are such that $\lambda -\lambda' \in \Z_\Om$, then $\lambda -\lambda' \in \Z$.
\item If $\lambda, \, \lambda' \in  \Z_\Om^3$ are such that $\lambda -\lambda' \in \Z_\Om$, then $\lambda -\lambda' \in \Z \cup \Z_\Om^1$.
\item If $\lambda  \in  \Z_\Om^3$ and $\, \lambda' \in \Z_\Om^2$ are such that $\lambda -\lambda' \in \Z_\Om$, then $\lambda -\lambda' \in \Z_\Om^3$. So by  (5) $\lambda' \in \Z_\Om^1$.
\end{enumerate}

From these facts we deduce that the density contribution from each of the sets $\Z, \,\, \Z_\Om^2$ to $\La$ can be at most $1$, and from $\Z_\Om^3$, at most $2$. But since $\rho(\La) = 4$, we must have $\rho(\La \cap \Z_\Om^2) =1$, and $\rho(\La \cap \Z_\Om^3) =2$. From $(6)$, it follows that $\La \cap \Z_\Om^2 = \La \cap \Z_\Om^1$, and then by (1), we get $\La \cap \Z_\Om^1 =  \Z_\Om^1 = \Z + 1/2$. Hence

$$\La = \Z \cup \Z + 1/2 \cup (\La \cap \Z_\Om^3).$$
Next, as $\Z_\Om^1 \neq \emptyset$,  $\lambda \in \La \cap \Z_\Om^3$ iff $2\lambda n,\,\,2\lambda m \in \Z + 1/2$. Let
$$2\lambda n = \frac{2k+1}{2}, \,\,\, 2\lambda m = \frac{2l+1}{2}$$
so if $j $ is the largest integer such that $2^j|n$, this is also true for $m$.

Finally we see that if $\Om$ is spectral, then $N = 2^{j+1}r +1,\,\, b= 2^{j+1}s$, with $r,s \in 2\Z+1$. Newman's theorem, applied to the set $S = \{0, 1, 2^{j+1}r +1, 2^{j+1}s\}, \,\, p=2,\,\, \alpha = 2$, ensures tiling of $\R$ by $\Om_4$. This completes the proof.

\section{\bf {Structure of the Spectrum for $n$ intervals} }

Let $\Omega =\cup_{j=1}^n [a_j,a_j +r_j),\,\, \sum_{j=1}^n r_j =1  $.  In this section we assume that $\Omega$ is a
spectral set with spectrum $\Lambda$. Recall that we may assume $0 \in \Lambda \subset \Lambda - \Lambda \subset \Z_\Omega.$

We will prove below that if, we assume that $\Lambda$  is a subset of some discrete lattice $\mathcal L$, then $\Lambda $ is rational. In one dimension, all known examples of spectra are  rational and periodic, though it is not known whether it has to be so. Given \ $\La$ has positive asymptotic density, if $\La \subset \mathcal L$, then $\La$ contains arbitrarily long arithmetic progressions by (Szemer$\grave{e}$di's theorem[S]). We now prove that in the $n$-interval case, as soon as $\La$ contains an AP of length $2n$, then the complete AP is also in $\La$. We first prove this for the set $\Z_\Om$:

\begin{proposition}\label{A}
If $\Z_{\Omega}$ contains an arithmetic progression of length $2n$ containing $0$, say $0, d, 2d, \cdots, (2n-1)d$ then
\begin{enumerate}
    \item  the complete arithmetic progression $d\Z \subset \Z_{\Omega}$,
  \item $d\in\Z$, and
  \item $\Omega$ d-tiles $\R$.
\end{enumerate}
\end{proposition}

\begin{proof}
Note that if $t \in \Z_\Omega $, then
$$ {\sum}_{j=1}^n  [e^{2\pi i t (a_j + r_j)} - e^{2\pi i t a_j}] \, = \,0.$$
The hypothesis says that $\widehat{\chi_{\Om}}(ld) = 0; \,\, l=1,2,...,2n-1$, hence
$$ {\sum}_{j=1}^n  [e^{2\pi i ld (a_j + r_j)} - e^{2\pi i ld a_j}] \, = \,0;\,\, l= 0,1,2,...,2n-1$$
We write $\zeta_{2j}=e^{2\pi i d a_j} \, , \, \,
  \zeta_{2j-1}= e^{2\pi i d (a_j + r_j)}; \, j=1,2,...,n$, then
the above system of equations can be rewritten as
$$\zeta_1^l-\zeta_2^l  \cdots+ \zeta_{2n-1}^l - \zeta_{2n}^l = 0; \,\, l = 0, 1, \cdots , 2n-1.$$

Equivalently,
\begin{equation}\label{2}
\left(\begin{array}{ccccc}
1 & 1 & \cdots & 1 & 1 \\
\zeta_1 & \zeta_2 & \cdots & \zeta_{2n-1} & \zeta_{2n}\\
\vdots & \vdots & \ddots & \vdots & \vdots \\
\zeta_1^{2n-2} & \zeta_2^{2n-2} & \cdots & \zeta_{2n-1}^{2n-2} & \zeta_{2n}^{2n-2} \\
\zeta_1^{2n-1} & \zeta_2^{2n-1} & \cdots & \zeta_{2n-1}^{2n-1} & \zeta_{2n}^{2n-1} \\
\end{array}\right)
\left(\begin{array}{c}
1\\ -1\\ \vdots \\ 1\\ -1\\
\end{array} \right)=
\left(\begin{array}{c}
0 \\ 0 \\ \vdots \\ 0 \\ 0\\
\end{array} \right)
\end{equation}

Since $ (1, -1,...,1, -1)^t \neq 0 $, we have

\begin{equation}\label{3}
\prod_{1=i<j}^{2n}(\zeta_i - \zeta_j)=0
\end{equation}

Next, we need a lemma.
\begin{lemma}
If (\ref{3}) holds, then there exist indices $i,\,j$ with $i \neq j$ and $i+j$ odd, such that $\zeta_i -\zeta_j =0$. We call such a pair $(\zeta_i,\zeta_j)$ a {\it good pair}. ( If $\zeta_i -\zeta_j =0$, with $i+j$ even, we say $(\zeta_i,\zeta_j)$ is a {\it bad pair}.)
 \end{lemma}
%%%%%%%%%%%%%%%%%
\begin{proof} Suppose not, then all solutions of (\ref{3}) are bad. Since the
solution set of (\ref{3}) is non-empty, without loss of generality let $\zeta_1 =\zeta_3$.
Then the system reduces to

\begin{equation}\label{4}
\left(\begin{array}{ccccc}
1 & 1 & \cdots & 1 & 1 \\
\zeta_2 & \zeta_3 & \cdots & \zeta_{2n-1} & \zeta_{2n}\\
\vdots & \vdots & \ddots & \vdots & \vdots \\
\zeta_2^{2n-3} & \zeta_3^{2n-3} & \cdots & \zeta_{2n-1}^{2n-3} & \zeta_{2n}^{2n-3} \\
\zeta_2^{2n-2} & \zeta_3^{2n-2} & \cdots & \zeta_{2n-1}^{2n-2} & \zeta_{2n}^{2n-2} \\

\end{array}\right)
\left(\begin{array}{c}
-\zeta_2\\ 2\zeta_3 \\ \vdots \\ \zeta_{2n-1}\\ -\zeta_{2n}\\
\end{array} \right)=
\left(\begin{array}{c}
0 \\ 0 \\ \vdots \\ 0 \\ 0\\
\end{array} \right)
\end{equation}

Hence

\begin{equation}\label{5}
\prod_{2=i<j}^{2n}(\zeta_i - \zeta_j)=0
\end{equation}

Note that any solution(good) of (\ref{5}) is a solution(good) of (\ref{3}).
Thus if (\ref{5}) has a good solution that would lead to a
contradiction. So (\ref{5}) does not have any good solution.
Repeating this $2n-2$ times we would be left with a $2\times 2$
Vandermonde matrix which is singular as
\begin{equation}\label{6}
\left(\begin{array}{cc}
1 & 1\\
\zeta_i & \zeta_j \\
\end{array}\right)
\left(\begin{array}{c}
n\zeta_i\\ -n\zeta_j \\
\end{array} \right)=
\left(\begin{array}{c}
0 \\ 0 \\
\end{array} \right)
\end{equation}
where $i$ is odd and $j$ is even. This implies that
$\zeta_i = \zeta_j $ with $i+j$  odd.
This is a  contradiction to our assumption.
\end{proof}

Now we can complete the Proof of Proposition \ref{A}.
 Observe that whenever we get a good solution of $(3)$, the system of equations in $(2)$ reduces to one involving a
$(2n-2)\times (2n-2)$ Vandermonde matrix. Now by the same arguments the reduced matrix would again
have a good solution. Repeating this process $n-1$ times, we get a partition of $\{\zeta_i\}$ into n distinct pairs
$(\zeta_i,\zeta_j)$ such that each $(\zeta_i,\zeta_j)$ is a good solution of  (3).
But then $\zeta_i^k = \zeta_j^k, \,\,\forall k \in \Z$. We can relabel the $\zeta_{2j}$'s, $j = 1,2,\cdots,n$ so that $\zeta_{2j-1} = \zeta_{2j}$. Then we get,
\begin{equation}
\widehat{\chi_\Omega}(kd)=\frac{1}{2\pi i
kd}\sum_{j=1}^n(\zeta_{2j-1}^{k}-\zeta_{2j}^k) =0;\,\, \forall \,k
\in \Z \setminus\{0\}
\end{equation}
Thus $d\Z \subset \Z_{\Omega}$.

Now consider
\begin{equation}\label{we}
F(x)=\sum_{k\in \Z} \chi_\Omega\left(x+k/d\right),\, x\in [0,1/d)
\end{equation}
Thus $F$ is $\frac{1}{d}$ periodic and integer valued
thus
\begin{equation}\label{62}
\widehat{F}(ld)= d \sum_{k \in \Z} \int_{0}^{\frac{1}{d}}
\chi_{\Omega}(x+k/d)e^{- 2\pi i ldx} dx =d
\widehat{\chi_\Omega}(ld)=d\, \delta_{l,0}
\end{equation}

Thus $F(t)=d \mbox{ a.e}$. So $ d\in \Z $ and $\Omega$  d-tiles the
real line.
\end{proof}

Using Proposition 4.1, we now prove the corresponding result for the spectrum.
%%%%%%%%%%%%%%%%%%%%%

\begin{theorem}
 Suppose $\Omega$ is spectral and $\Lambda$ a spectrum
with $0\in \Lambda$. If for some $a,d\in \R,\, a,a+d,...,a+(2n-1)d \in \Lambda $,
then $a+d\Z \subseteq \Lambda $. Further $d \in \Z$ and $\Om \,\,d$-tiles $\R$.
\end{theorem}
%%%%%%%%%%%%%%%%%%%%%%%%%%%%%%%%%%%
\begin{proof}
 Since $$  a,a+d,...,a+(2n-1)d \in \Lambda$$
Shifting $\Lambda$ by $a$ we get $\Lambda_1=\Lambda-a$ is a spectrum
for $\Omega$ and
$$0,d,...,(2n-1)d \in \Lambda_1 \subset \Lambda_1 - \Lambda_1 \subset \Z_{\Omega}.$$
Thus surely $d\Z \subset \Z_\Omega $ by Proposition \ref{A}.
Now let $\lambda \in \Lambda_1$. Then by orthogonality,
 $$-\lambda,d-\lambda,2d-\lambda,...,(2n-1)d-\lambda \in \Z_{\Omega}.$$
Let $$\xi_{2j} = e^{-2 \pi i \lambda a_j} , \xi_{2j-1} = e^{-2 \pi i \lambda(a_j+r_j)}; j=1,...,n$$
$$\zeta_{2j}=e^{2\pi i d a_j}, \zeta_{2j-1}=e^{2\pi i d(a_j+r_j)}; j=1,...,n$$

Since $\widehat{\chi_\Om}(kd-\lambda) = 0, \, k= 0,...,2n-1 $, we have
\begin{equation}\label{7}
\xi_1 \zeta_1^k - \xi_2 \zeta_2^k + ...+\xi_{2n-1} \zeta_{2n-1}^k - \xi_{2n} \zeta_{2n}^k =
 0 \mbox{ for } k=0,...,2n-1.
\end{equation}

Since the ${\zeta_i}'$s can be partitioned into n disjoint pairs
$(\zeta_i,\zeta_j)$ such that $\zeta_i=\zeta_j$ and $i+j$ odd,
 without loss of generality, we relabel the $\zeta_{2j}$'s and the corresponding $\xi_{2j}$'s  so that
$\zeta_{2j-1}=\zeta_{2j},\, j=1,2,...,n$.
Thus (\ref{7}) can be written as

\begin{equation}\label{8}
\left(\begin{array}{cccc}
1 & 1 & \cdots & 1  \\
\zeta_1 & \zeta_3 & \cdots & \zeta_{2n-1} \\
\vdots & \vdots & \ddots & \vdots  \\
\zeta_1^{n-1} & \zeta_3^{n-1} & \cdots & \zeta_{2n-1}^{n-1}  \\
\end{array}\right)
\left(\begin{array}{c}
\xi_1-\xi_2\\ \xi_3-\xi_4 \\ \vdots \\ \xi_{2n-1} -\xi_{2n}\\
\end{array} \right)=
\left(\begin{array}{c}
0 \\ 0 \\ \vdots \\ 0
\end{array} \right)
\end{equation}

\medskip
Now if $[ \xi_1 - \xi_2,\xi_3-\xi_4,...,\xi_{2n-1}-\xi_{2n}]^t $ is the trivial solution
i.e. $\xi_{2j-1}-\xi_{2j}= 0,\,  \forall \,j=1,...,n$,
then $\forall \,k \in \Z ,$
\begin{eqnarray*}
    \widehat{\chi_\Omega}(kd-\lambda) & = & \frac{1}{2\pi i (kd-\lambda)} \left[ \xi_1 \zeta_1^k - \xi_2\zeta_2^k+ \cdots+\xi_{2n-1}\zeta_{2n-1}^k -\xi_{2n} \zeta_{2n}^k \right ]\\
  & = & \frac{1}{2\pi i (kd-\lambda)} \left[\zeta_1^k (\xi_1-\xi_2) +\cdots+\zeta_{2n-1}^k (\xi_{2n-1}-\xi_{2n})\right]=0
\end{eqnarray*}
Thus $d \Z -\lambda \in \Z_\Omega$.

If however $[\xi_1-\xi_2,\xi_3-\xi_4,...,\xi_{2n-1}-\xi_{2n}]^t$ is not the trivial solution,
then $\zeta_{2l-1}=\zeta_{2k-1}$ for some $l,k \in 1,...,n ; l\neq k.$

Removing all the redundant variables and writing the remaining variables as $\eta _{2j+1}^l , \,\, j,l  = 0, 1 \cdots k-1$, we get a non-singular Vandermonde
matrix satisfying

\begin{equation}\label{22}
\left(\begin{array}{cccc}
1 & 1 & \cdots & 1 \\
\eta_1 & \eta_3 & \cdots & \eta_{2k-1} \\
\vdots & \vdots & \ddots & \vdots \\
\eta_1^{k-1} & \eta_3^{k-1} & \cdots & \eta_{2k-1}^{k-1} \\
\end{array}\right)
\left(\begin{array}{c}
\sum_1\\ \sum_3 \\ \vdots \\ \sum_{2k-1}\\
\end{array} \right)=
\left(\begin{array}{c}
0 \\ 0 \\ \vdots \\ 0 \\
\end{array} \right)
\end{equation}

where $$ {\sum}_{k} = \sum_{j:{\zeta}_{2j-1}=\eta_k} {{\xi}_{2j-1} -
{\xi}_{2j}}. $$ Then each of the  $\sum_i =0 , i=1,\cdots,k$. But
then once again $\forall p  \in  \Z$
$$\widehat{\chi_\Omega}(pd-\lambda)=\frac{1}{2\pi i (pd-\lambda)} \left[ \eta_1^p {\sum}_1 + \eta_3^p{\sum}_3+\cdots
+ \eta_{2k-1}^p{\sum}_{2k-1}\right]=0   $$ Thus $d\Z -\lambda  \in
{\Z}_{\Omega}$. We already have $d\Z \subseteq \Z_{\Omega}$ and now
we have seen if $\lambda \in \Lambda_1$, then $d\Z -\lambda \in
\Z_\Omega $. Thus $d\Z \subseteq \Lambda_1$, hence $a+d\Z \subset
\Lambda$.
\end{proof}
\medskip

\subsection{Existence of APs.}With the above Proposition it is clear that we need to explore conditions under which the spectrum, when it exists, contains arithmetic progressions.  In the following lemma, we are able to relax the condition that $\La$ be contaned in a lattice condition to a local condition. We say that a discrete set $\Delta$ is $\delta-$ separated, if $\inf\{|\lambda - \lambda'|: \lambda,\, \lambda' \in \Delta, \, \lambda \neq \lambda'\} = \delta > 0$.

\begin{lemma}.
If $\La$ is a spectrum, such that $\Gamma = \La - \La$ is $\delta$-separated, then $\La$ contains a complete arithmetic progression. Further, $\La$ is contained in a lattice with a base.
\end{lemma}

\begin{proof}:
The $\delta$-separability of $\Gamma$ means that for some $\delta >0$
$$ inf \{ |\gamma_j -\gamma_k| : \gamma_j,\, \gamma_k \in \Gamma, \,\, \gamma_j \neq \gamma_k \} = \delta >0 $$

Let $\delta_1 = \delta/2$ and consider the  map from $\La$ to the lattice $\delta_1 \Z$ given by
$$\psi(\lambda) = [\lambda/\delta_1] \delta_1,$$
where $[x]$ denotes the largest integer less than $x$.

Then the subset $\psi(\La)$ has positive density, hence by [S], given any N, there exists an arithmentic progression of length $N$, say
$$\psi(\lambda_1), \psi(\lambda_2), ... ,\psi(\lambda_N),$$ with $\psi(\lambda_{j+1}) =\psi(\lambda_j) + d\delta_1,\,\, d\in\Z,\,\, j= 1,2,...N-1$.
But then $\lambda_{j+1} -\lambda_j $ must lie in the interval $(d\delta_1 -\delta_1, d\delta_1 + \delta_1)$ for each $j= 1,2,...,N-1$. Since this interval has length $\delta$, all the above elements must be the same (in any interval of length $\delta $, there can be at most one element of $\La -\La$). But this means that the $\lambda_j$'s are in arithmetic progression. Now using Proposition 4.3, we get that $\La$ contains the complete AP $M_d\Z$, where $M_d = d\delta_1$, and $M_d \in \Z$.

Let $\Lambda_s := \{\lambda_{n+1}-\lambda_n | \lambda_n \in \Lambda \}$
be the set of successive differences of spectral sequences, the spectral gaps. We will show that $\Lambda_s$ is finite.
$\Lambda_s \subseteq \Lambda - \Lambda \subseteq \Z_\Omega $. As $\Omega$ is measurable and of finite measure there exists a neighbourhood
 around $0$ which does not
intersect $\Z_\Omega$. Thus $\Lambda_s$ is bounded below . But $\Lambda_s$ is bounded above [see IP].
So by the compactness of $\Omega$ we get that $\Lambda_s$ is finite as
$\widehat{\chi_\Omega}$ can be extended
analytically to the entire complex plane and zeros of an entire function are isolated.

Let $\Lambda_s = \{r_1,r_2,...,r_k\}$.
The set of solutions for
$\sum_{i=1}^{k} a_i r_i = M_d$ with $a_i \in \N \cup \{0\}$ is finite. Thus $\Lambda \subseteq \{0,b_1,...,b_l\}+ M_d\Z$ for some $l$.
In other words, $\Lambda$ is contained in a lattice with a base.
So we have $M_d\Z \subseteq \Lambda \subseteq \{0,b_1,...,b_l\}+ M_d\Z$
\end{proof}

We end this section by showing  that the hypothesis $\La -\La \subseteq \mathcal L$ gives more information on the spectrum.
\begin{theorem}
If $\Omega$ as above is spectral,
with spectrum $\Lambda$ such that
$0\in \Lambda \subset \Lambda-\Lambda \subset \mathcal{L}$, $\mathcal{L}$
a lattice, then $\Lambda$ is rational.
\end{theorem}

\begin{proof}:
Let $\mathcal{L} = \theta \Z$, and suppose $\Lambda \subseteq \theta \Z$.
Since $\Lambda$ is a spectrum it has asymptotic density 1.
By Szemer$\grave{e}$di's theorem $\Lambda$ contains arbitrarily long arithmetic progressions, and so in
particular, of length 2n. Without loss of generality, suppose
$$0,K\theta,...,(2n-1)K\theta \in \Lambda \subseteq \Z_\Omega$$
where $K \in  \Z$. We already know that if $\Lambda$ contains an arithmetic
progression of length $2n$ then the complete arithmetic progression is in $\Lambda$
and the common difference $d \in \N$. Hence $K\theta \in \Z$, which means that $\theta \in \Q$. So $\Lambda \subseteq \Q$.
\end{proof}

\section{Three Intervals: Spectral implies Tiling}

Consider now the particular case of three intervals. Let $\Omega = [0,r_1) \cup [a_2,a_2 +r_2) \cup [a_3 +r_3),\,\, r_1 +r_2 +r_3 = 1$.
Suppose that $\Omega $ is spectral and that its spectrum $\Lambda$ contains an arithmetic progression of length 6
and common difference d. By translating if
necessary we may assume that $d\Z \subset \Lambda , \,\, d \in \Z$, and so by Proposition 4.1, $\Om \,\, d-$tiles $\R$. Thus,
if $d=1, \,\, \Omega$ tiles $\R$ by $\Z$.
If $d\neq 1 $, let $ \lambda \in \Lambda\setminus d\Z $. Put
$$\zeta_{2j-1}=e^{2\pi i d(a_j+r_j)},\zeta_{2j}=e^{2\pi i d a_j}$$
$$\xi_{2j-1}=e^{2\pi i\lambda (a_j+r_j)},\xi_{2j} =e^{2\pi i \lambda a_j},$$
for $j = 1, 2, 3$. Now since $\lambda, \,\, 0,d,\cdots , 5d \in \La$, by orthogonality we get,
$$\xi_1\zeta_1^k - \xi_2\zeta_2^k + \cdots + \xi_{2n-1}\zeta_{2n-1}^k - \xi_{2n}\zeta_{2n}^k = 0$$
for $k = 0,1,\cdots,5$.

We know that among the $\zeta_j$'s there are good
pairs. We can reindex the $\zeta_{2j}$'s and simultaneously the corresponding $\xi_{2j}$'s such that $\zeta_{2j-1}=\zeta_{2j}$. Then we have

\begin{equation}\label{223}
\left(\begin{array}{ccc}
1 & 1 & 1 \\
\zeta_1 & \zeta_3 & \zeta_5 \\

\zeta_1^2 & \zeta_3^2 & \zeta_5^2 \\
\end{array}\right)
\left(\begin{array}{c}
\xi_1 -\xi_2 \\ \xi_3-\xi_4 \\ \xi_5-\xi_6\\
\end{array} \right)=
\left(\begin{array}{c}
0 \\ 0 \\  0 \\
\end{array} \right)
\end{equation}

Let $A$ denote the linear transformation represented by the above matrix. We consider three cases:

\medskip

{\bf Case 1}. Rank A=3.
In this case, $\xi_1=\xi_2,\, \xi_3=\xi_4$ and $\xi_5=\xi_6$.
But then it follows that $\Lambda$ is a group. Further, $\lambda \in \La$ implies $n\lambda \in \La \,\, \forall \,\, n\in \Z$. This means that $\Om \,\, d-$tiles $\Z$, so infact, $\lambda \in \Z$, i.e. $\La = k\Z$ for some $k \in \N$. But $\Lambda$ has density $1$, we must have $\Lambda = \Z$.

\medskip

{\bf{Case 2}}. Rank A=2.
In this case, without loss of generality, assume that $\zeta_1=\zeta_5$, then

\begin{equation}\label{224}
\left(\begin{array}{cc}
1 & 1  \\
\zeta_1 & \zeta_3  \\
\end{array}\right)
\left(\begin{array}{c}
\xi_1 -\xi_2+\xi_5-\xi_6 \\ \xi_3-\xi_4 \\
\end{array} \right)=
\left(\begin{array}{c}
0 \\ 0 \\
\end{array} \right)
\end{equation}
\\
Hence $$\xi_1-\xi_2+\xi_5-\xi_6=0 \mbox{  and  } \xi_3-\xi_4=0$$
It follows that in this case
$$\Z_\Omega=\Z_{\Omega}(1)\cup\Z_{\Omega}(2)\cup\Z_\Omega(3)$$
where
$$\Z_{\Omega}(1)=\{\lambda:\xi_3=\xi_4,\xi_1=\xi_2,\xi_5=\xi_6\}$$
$$\Z_{\Omega}(2)=\{\lambda:\xi_3=\xi_4,\xi_1=\xi_6,\xi_2=\xi_5\}$$
$$\Z_{\Omega}(3)=\{\lambda:\xi_3=\xi_4,\xi_1=-\xi_5,\xi_2=-\xi_6\}.$$

We proceed as in ([L1]) to conclude that
\begin{enumerate}
\item Either $ \Lambda \subset \Z_{\Omega}(1)\cup \Z_{\Omega}(3)$
or $\Lambda \subset \Z_{\Omega}(2)\cup \Z_{\Omega}(3)$,
\item $\Lambda \subset \Lambda-\Lambda \subset
 \Z_{\Omega}(3) \cup (\Z_{\Omega}(1)\cap \Z_{\Omega}(2))$
\item $\Z_{\Omega}(1)\cap \Z_{\Omega}(2)=k\Z, \,\, k \in \Z$, since this set is a subgroup of  $ \Z$, and
\item Let $\lambda_1,\, \lambda_2 \in \La \cap \Z_\Om(3)$, then $\lambda_1 -\lambda_2 \in (\Z_\Om(1) \cap \Z_\Om(2)) \setminus \Z_\Om(3).$
\end{enumerate}

It follows that $\Lambda \cap \Z_{\Omega}(3) \subseteq \beta+k\Z$. Then by density considerations, and Theorem 4.3, we get that $ \Lambda = k\Z \cup (k\Z+\beta)$, with $ k=2$. Then {\it Spectral implies Tiling} follows from [P2], where it is proved that if the spectrum is a union of two lattices, then $\Om$ tiles $\R$.

\medskip
{\bf Case 3}. Rank A=1\\
In this case,  $\zeta_1=\zeta_3=\zeta_5=\zeta_2=\zeta_4=\zeta_6$. Then
$$ e^{2 \pi i d a_1}= e^{2 \pi i d (a_1+r_1)}=e^{2 \pi i d a_2}
=e^{2 \pi i d (a_2+r_2)}= e^{2 \pi i d a_3}= e^{2 \pi i d (a_3+r_3)}.$$
Taking $a_1 = 0$ we get
$$ a_1 = 0,a_2=\frac {l_2}{d}, a_3=\frac{l_3}{d} ; r_1=\frac {k_1}{d}, r_2=\frac{k_2}{d},
r_3=\frac{k_3}{d}$$
where $l_2,l_3,k_1,k_2,k_3 \in \Z $ and $k_1 + k_2 + k_3 = d$.

It follows from this that we are in the case of $d$ equal intervals, in three groups.
But then  the spectrum is periodic and of the form
$$\Lambda = L + d\Z$$
If $d=1$ then $\Lambda = \Z$, so there is nothing to prove. If $d =2$, then $r_1,r_2,r_3 \geq 1/2$, which is not possible. Hence, we may assume that $d \geq 3$. We beleive that $d =3$ in this case, but we are unable to prove this. However, if we make a further assumption such as $\La \subset \mathcal L$, a lattice, or that $\La \subset \Q$, then we are led to questions of vanishing sums of roots of unity given below.
Let $\La \subset \Q$
$$e^{2\pi i \lambda a_1} -1 +
 e^{2\pi i \lambda (a_2+r_2)} -e^{2\pi i \lambda a_2} +
 e^{2\pi i \lambda (a_3+r_3)}-e^{2\pi i \lambda a_3} = 0$$
 and this is a case of six roots of unity summing to $0$, say
$\alpha_1+\cdots+\alpha_6 = 0$. Poonen and Rubinstein [PR] ( have classified all minimal vanishing sums of roots of unity
$\alpha_1+\cdots+\alpha_n = 0$ of weight $n\leq 12$ (see also[LL]).
There are three  possible ways in which six roots of unity $(\alpha_1, \alpha_2, ... , \alpha_6)$ can sum up to zero. If $\sigma $ denotes an element of $S_6$, the group of permutations of 6 objects, the possible cases are:

\begin{enumerate}
\item[(T1)] ( a $2$ sub-sum $=0$). $\alpha_{\sigma(1)}+\alpha_{\sigma(2)}= \alpha_{\sigma(3)}+\alpha_{\sigma(4)}= \alpha_{\sigma(5)}+\alpha_{\sigma(6)}=0$.
\item[(T2)] (a $3$ sub-sum $=0$). $\alpha_{\sigma(1)}+\alpha_{\sigma(2)}+\alpha_{\sigma(3)}= \alpha_{\sigma(4)}+\alpha_{\sigma(5)}+\alpha_{\sigma(6)}=0$.

\item[(T3)] (no sub-sum $=0$). $\alpha_{\sigma(n)}=\rho^n ;\,\, n=1,\cdots,4 ; \,\,\alpha_{\sigma(5)}=-\omega ,\alpha_{\sigma(6)}=-\omega^2$ (after normalizing ) where $\rho$ is a fifth root of unity and $\omega $ is a cube root of unity.
\end{enumerate}

It turns out that there are now many possibilities and we use a symbolic computation explained in the next section.

\section {A symbolic computation using Mathematica}

We begin with the setting of Case 3 in the previous section, where Rank$ A = 1$. Let $d$ be the smallest positive integer in $\Z_\Om$ such that $d\Z \subseteq \La$, and let $\La = L + d\Z = \cup_{j=1}^{d} \{\lambda_j + d\Z\}$, with $\lambda_1 = 0$. We also have $ d \geq 3$. Observe that if $d' \in \Z_\Om$ is such that $e^{2 \pi i d' a_j}= e^{2 \pi i d' (a_j+r_j)} = 1$ for all $j \in \{1,2,3\}$, then $d' \in d\Z$.
\medskip

We now introduce some notation and explain the analysis behind  the computation carried out. First we map $\Z_\Om$ into $\C^6$;
$ \lambda\rightarrow v_\lambda=(e^{2 \pi i \lambda (a_1 + r_1)},- e^{2 \pi i \lambda a_1},e^{2 \pi i \lambda(a_2 + r_2)},-e^{2 \pi i \lambda a_2}, e^{2 \pi i \lambda (a_3 + r_3)},-e^{2 \pi i \lambda a_3})$.
In particular $v_0 = (1, -1, 1, -1, 1, -1)$
\medskip

Define a conjugate bilinear form on $\C^6$ as follows. For $v = (x_1,x_2,...,x_6)$, $ w = (y_1,y_2,...,y_6)$, (skew dot product)
$$ SDP(v,w) = x_1 \bar y_1 - x_2 \bar y_2 + x_3 \bar y_3 - x_4 \bar y_4 + x_5 \bar y_5 - x_6 \bar y_6 $$
Note that $SDP(v_\lambda ,v_0) = 0 \,\, \forall \lambda \in \Z_\Om$. Let
$$G(v,w) = (x_1 \bar y_1, - x_2 \bar y_2,  x_3 \bar y_3, - x_4 \bar y_4 , x_5 \bar y_5, - x_6 \bar y_6)$$
We will say that a vector $v_\lambda $ is of Type1, Type2 or Type3 if it satisfies (T1), (T2) or (T3) respectively, listed at the end of the last section.

\medskip
We make some observations
\begin{enumerate}
    \item $SDP(v_\lambda, v_0) = 0 ,\,\,\forall \lambda \in \Z_\Om$.
    \item $G(v_{\lambda_1},v_{\lambda_2}) = v_{\lambda_1 -\lambda_2}$.
    \item If $\lambda \subset \Q$, then all components of $v_{\lambda}, \,\, \lambda \in \lambda$ are roots of unity.
    \item Since $a_1 = 0$, the second coordinate in the image of $\Z_\Om$ in $\C^6$ is always $-1$.
    \item The image of $\La$ in $\C^6$ consists of precisely $d$ elements corresponding to the different cosets of $d\Z$ (for, if $v_{\lambda_1} =v_{\lambda_2}$, then $G(v_{\lambda_1}, v_{\lambda_2}) = v_0$, so $\lambda_1 -\lambda_2 \in d\Z$.
\end{enumerate}
 \medskip

The computation is done under the following assumption:

\medskip

{\bf Assumption:  For $\lambda_1,\, \lambda_2 \in \La, \,\,  v_{\lambda_1 - \lambda_2}$ is of Type 1  if and only if $\lambda_1 -\lambda_2 \in d\Z$}

Through the symbolic computation we shall investigate the maximum possible cardinality of a set $\{\lambda_1, \lambda_2,...,\lambda_d\}$ such that $\lambda_i -\lambda_j \in \Z_\Om$  and such that $v_{\lambda_i -\lambda_j}$ is a vector of either Type 2 or Type 3.
\medskip

The first case of this investigation is analyzed below.

 {\bf Case 1.} $v_{\lambda_1}, \, v_{\lambda_2} $ are both Type 2 vectors.
\medskip
First we make some observations.

1. if $(\alpha_1,\alpha_2,...,\alpha_6)$ are 6 roots of unity such
that their sum is zero and if $\alpha_{\sigma(1)}+\alpha_{\sigma(2)}+\alpha_{\sigma(3)}=0 $ then
$\frac{\alpha_{\sigma(1)}}{\alpha_{\sigma(2)}}$,
$\frac{\alpha_{\sigma(2)}}{\alpha_{\sigma(3)}}$,
$\frac{\alpha_{\sigma(3)}}{\alpha_{\sigma(1)}}$ are powers of
$\omega$ where $\omega^3=1$.\\

2. if $(\alpha_1,\alpha_2,...,\alpha_6)$ is as above and has no subsum
zero i.e. it is a Type3 vector, then it has to be a permutation of
$(x \rho, x \rho^2,x \rho^3, x \rho^4,-x \omega,\\ -x \omega^2)$ where
$x$ is some root of unity. So there cannot be two pairs among these six
elements whose ratios are powers of $- \omega$.\\

3. If the vector $(\pm \omega^{\ast},\pm \omega^{\ast},\pm
\omega^{\ast},\pm \omega^{\ast},\pm \omega^{\ast},\pm
\omega^{\ast})$ where $\omega^{\ast}$ is some power of
 $\omega,$ with exactly 3 positive signs and 3 negative signs has sum
 zero then it has to be of Type1, because Type2 would imply that
 it is a permutation of $(1, \omega, \omega^2, -1, -\omega,
 -\omega^2 )$ which is a Type1 vector.\\

 4. For a Type2 vector all elements in a 3 subsum adding up to zero
 have the same sign.

\medskip
 Let $u[x]=(1, \omega, \omega^2, x, x\omega, x\omega^2) $ and
     $u[y]=(1, \omega, \omega^2,y , y\omega, y\omega^2) $ be two
     vectors of Type2 where $x$ and $y$ are roots of unity. Take the
     conjugate SDP of these two vectors by permuting and conjugating
     the second vector $u[y]$. The terms in the conjugate SDP,
     ignoring the signs, will fall into one of the following four
     categories upto a permutation of the first 3 elements and last
     3 elements.

\medskip
 1. $(\omega^\ast,\omega^\ast,\omega^\ast, x \bar{y} \omega^\ast, x
 \bar{y} \omega^\ast, x \bar{y} \omega^\ast)$

 2. $( \bar{y} \omega^\ast,
 \bar{y} \omega^\ast,  \bar{y} \omega^\ast, x \omega^\ast,x \omega^\ast,x \omega^\ast)$

3. $(\omega^\ast,\omega^\ast,\bar{y} \omega^\ast, x \omega^\ast, x
\bar{y} \omega^\ast, x
 \bar{y} \omega^\ast)$

4. $(\omega^\ast,\bar{y} \omega^\ast,\bar{y} \omega^\ast, x
\omega^\ast, x \omega^\ast, x
 \bar{y} \omega^\ast)$

\medskip
Here $\omega^\ast$ represents some power of $\omega$. The first case
arises when we take conjugate SDP of u[x] and permuted u[y] which is
of the type $(\omega^\ast,\omega^\ast,\omega^\ast, y \omega^\ast, y
\omega^\ast, y \omega^\ast)$. The second case is similar. The third case arises when we take the conjugate SDP
of $u[x]$ and permuted $u[y]$ which is of the type $(\omega^\ast,
\omega^\ast, y \omega^\ast, \omega^\ast, y \omega^\ast,y
\omega^\ast)$ upto a permutation of the first 3 elements and last 3
elements. The fourth case is again similar.

\medskip
In all cases, after putting in the signs, they are not of type 3, as
there are atleast two pairs whose ratios are powers of $-\omega$.

\medskip
Now we need to consider only two cases.

\medskip
Case 1. $(\omega^\ast,\omega^\ast,\omega^\ast, x \bar{y}
\omega^\ast, x \bar{y} \omega^\ast, x \bar{y} \omega^\ast)$

If there is a 3 subsum being zero, after the signs are put in
appropriately, which involves $\omega^\ast$ and $x \bar{y}
\omega^\ast$ then the ratio $(x \bar{y} \omega^\ast / \omega^\ast) =
x \bar{y}$ is a power of $ - \omega$. Hence the set
$(\omega^\ast,\omega^\ast,\omega^\ast, x \bar{y} \omega^\ast, x
\bar{y} \omega^\ast, x \bar{y} \omega^\ast)$ becomes $(\pm
\omega^{\ast},\pm \omega^{\ast},\pm \omega^{\ast},\pm
\omega^{\ast},\pm \omega^{\ast},\pm \omega^{\ast})$ with exactly 3
positive signs and 3 negative signs. So the terms in this SDP form a
Type1 vector. Hence this possibility is not considered. So if there
is a 3 subsum being zero then it should be
$\omega^{\ast}+\omega^{\ast}+\omega^{\ast}=0$. and $x \bar{y}
\omega^\ast+ x \bar{y} \omega^\ast+ x \bar{y} \omega^\ast=0$. The 3
pluses occur with first 3 elements and 3 minuses occur with the last
3 elements or vice versa. Once u[x] is fixed with these signs $(1,
\omega, \omega^2, y, y \omega, y \omega^2)$ can be permuted in the
first 3 and last 3 elements. In this case x and y can be any root of
unity other than $-1,-\omega,-\omega^2$.\\

Case 3. $(\omega^\ast,\omega^\ast,\bar{y} \omega^\ast, x
\omega^\ast, x \bar{y} \omega^\ast, x
 \bar{y} \omega^\ast)$

 If there is a 3 subsum being zero after the signs are put in
 appropriately, then the 3 subsum which involves $\omega^\ast$ has
 to involve one of the terms $\bar{y} \omega^\ast, x \omega^\ast, x
 \bar{y} \omega^\ast$. So either $x$ is a power of $-\omega$ or $y$
 is a power of $-\omega$ or $x \bar{y}$ is a power of $- \omega$. If
 $x \bar{y}$ is a power of $- \omega$ then both $
 \bar{y}\omega^\ast, x \omega^\ast$ will also be involved in a 3
 subsum which has $\omega^\ast$. So both $x$ and $y$ are powers of
 $- \omega$. Hence the set $(\omega^\ast,\omega^\ast,\bar{y} \omega^\ast, x
\omega^\ast, x \bar{y} \omega^\ast, x
 \bar{y} \omega^\ast)$ becomes $(\pm
\omega^{\ast},\pm \omega^{\ast},\pm \omega^{\ast},\pm
\omega^{\ast},\pm \omega^{\ast},\pm \omega^{\ast})$  with exactly 3
positive signs and 3 negative signs. So the terms in this SDP form a
Type1 vector. Hence this possibility is not considered. Hence
either $x$ is a power of $- \omega$ or $y$ is a power of $- \omega$
but not both.\\

If $x$ is a power of $- \omega$ then $x$ is a power of $\omega$ as
$u[x]$ is a Type2 vector and the 3 subsum would be
$\omega^\ast+\omega^\ast+\omega^\ast=0$. and $\bar{y} \omega^\ast +
\bar{y} \omega^\ast+ \bar{y} \omega^\ast = 0$. The 3 pluses occur
with elements which involve $\bar{y}$ and 3 minuses occur with the
other 3 elements or viceversa. So, keeping $u[y] =
(1,\omega,\omega^2,y, y \omega, y \omega^2)$ fixed, the 3 pluses
occur with first 3 elements and 3 minuses occur with the last three
elements or viceversa. The vector $u[x]=(1,\omega,\omega^2,x, x
\omega, x \omega^2)$ with $x= \omega^\ast$ is itself a permutation
of $(1,\omega,\omega^2,1, \omega, \omega^2)$. Now $u[x]$ has to be
permuted in such a way that the terms in the SDP which involve
$\bar{y}$ have to be a permutation of $\pm (\bar{y}, \bar{y} \omega,
\bar{y} \omega^2)$. Hence permuted $u[x]$ has to be of the form $(
\sigma(1), \sigma(omega), \sigma(\omega^2), \mu(1), \mu(\omega),
\mu(\omega^2))$ where $\sigma$ and $\mu$ are permutations of
$(1,\omega,\omega^2)$. In this case $y$ can be any root of unity
other than $ -1, - \omega, -\omega^2$.\\

The analysis when y is a power of $- \omega$ is identical.

So all the cases where the two Type2 vectors $u[x]$ and $u[y]$ can
 have conjugate SDP zero and the terms in the SDP forms a Type2
 vector, reduce to the case where the 3 pluses occur at the first 3
 positions and 3 minuses occur at the last 3 positions or viceversa,
 with $u[x]$ fixed and $u[y]$ permuted among the first 3 positions
 and last 3 positions or where all the permuted $y$ terms in $u[y]$
 occur in the first 3 positions and other permuted 3 terms occur at
 the last 3 positions, as given below

\begin{equation}
\left(\begin{array}{cccccc}
+1 & +1 & +1 & -1 & -1 & -1\\
1 & \omega & \omega^2 & x & x\omega & x \omega^2\\
\sigma(1) & \sigma(\omega) & \sigma(\omega^2) & y \mu(1) & y \mu(\omega) & y\mu(\omega^2)\\
\end{array} \right)
\end{equation}

 It follows that if the spectrum contains only vectors of Type1 and Type2, then $d\leq 3$. This follows from the computation given at the end of\\
http://www.imsc.res.in/ $\sim $ rkrishnan/FugledeComputation.html - \\
 Type2withType2.nb
 \bigskip

Now in the case when $\La$ contains one vector of Type3, Two cases can arise:

1. $\La$ contains only vectors of Type1 and Type3, then from the Mathematica symbolic computation, available in the file\\ http://www.imsc.res.in/$\sim $rkrishnan/FugledeComputation.html - \\
Type3withType3.nb
we conclude again the $d = 3$.

\medskip
2. Lastly if $\La$ contains vectors of Type 3 as well as of Type 2, then there can be at most one coset coming from each type. The details of this computation are available at\\ http://www.imsc.res.in/$\sim $rkrishnan/FugledeComputation.html - \\
Type3withType2.nb and vwithuandv1.nb

\medskip
We conclude that $d =3$ in all the above cases, which reduces to the case of three equal intervals, for which we have already proved the conjecture.

 We have thus proved the spectral implies tiling part of Fuglede's conjecture for three intervals under two assumptions, namely (a)
$\La$ is contained in a lattice, and (b) $(\La -\La) \cap$ Type1 $= d\Z$.

\medskip
{\bf Final Remark.} Note that the additional assumptions made on the spectrum are used to reduce the Rank $A =1$ case to the case of three equal intervals.  However, without any additional assumptions on $\Lambda$, this case still corresponds to an equal interval case grouped together in three bunches.

\bigskip
{\bf Acknowledgement.} The first author would like to thank Biswaranjan Behera for useful discussions during the initial stages of this work.

\end{document}